\documentclass{amsart}[10pt] 
\usepackage{amscd, amsmath, amssymb, amsthm, enumerate, mathrsfs, hyperref}

\title{Some new surfaces of general type with maximal Picard number} 
\author{Partha Solapurkar}
\address{Department of Mathematics\\
Purdue University\\
West Lafayette, IN 47907\\
U.S.A.}
\email{psolapur@purdue.edu}

\date{}  

\newtheorem{theorem}{Theorem}

\newtheorem{proposition}[theorem]{Proposition}
\newtheorem{lemma}[theorem]{Lemma}

\DeclareMathOperator{\NS}{NS}

\DeclareMathOperator{\tr}{tr}
\DeclareMathOperator{\Gr}{Gr}

\DeclareMathOperator{\Sp}{Sp}
\DeclareMathOperator{\Hom}{Hom}
\DeclareMathOperator{\End}{End}

\newcommand {\CC} {\mathbb C}
\newcommand {\ZZ} {\mathbb Z}
\newcommand {\RR} {\mathbb R}
\newcommand {\QQ} {\mathbb Q}
\newcommand {\HHH} {\mathbb H}
\newcommand {\HH} {\mathfrak H}

\newcommand {\EE} {\mathcal E}
\newcommand {\OO} {\mathcal O}
\newcommand {\OOO} {\mathscr O}
\newcommand {\sF} {\mathscr F}
\newcommand {\XX} {\mathcal X}
\newcommand {\WW} {\mathcal W}

\newcommand {\AAA} {\mathcal A}

\begin{document}
\maketitle

\begin{abstract}
  We construct some complex surfaces of general type with maximal Picard number. These examples arise as fibrations of genus two curves over quaternionic Shimura curves. 
\end{abstract}

\section{Introduction}
The Neron-Severi group of an algebraic surface $ X $ is the group $ \NS(X) $ of divisors on $ X $ modulo algebraic equivalence. The rank $ \rho(X) $ of $ \NS(X) $ is called the {\em Picard number} of $ X $. The first Chern class gives an isomorphism
\[ c_1 : \NS(X) \otimes_\ZZ \QQ \overset{\sim}\longrightarrow H^2(X, \QQ) \cap H^{1,1}(X). \]
This gives an obvious upper bound on the Picard number:
\[ \rho(X) \leq h^{1,1}(X). \]
We are looking for \emph{nontrivial} examples of surfaces which achieve this bound. We will say that such a surface is {\em Picard maximal}. A trivial set of examples is given by surfaces with geometric genus zero, since in that case
\[ H^2(X, \QQ) \cap H^{1,1}(X) = H^2(X, \QQ). \]

In \cite{shioda}, Shioda constructs elliptic modular surfaces and among other things, proves that they are Picard maximal. Roughly speaking, our construction tries to make something similar work in the case of genus two fibrations.

\subsection*{Acknowledgements}
I would like to thank my advisor Professor Donu Arapura for numerous conversations and guidance throughout this project. I would also like thank Professor Ben McReynolds for his suggestions regarding Shimura/Teichm\"uller curves in the moduli space of curves. My thanks also to Nicholas Miller for helping me with quaternion algebras. 

\section{Hodge theory of fibered surfaces}
Let $ X $ be a connected nonsingular complex projective algebraic surface, let $ C $ be a connected nonsingular complex projective curve and let $ f : X \to C $ be a nonconstant projective algebraic morphism with connected fibers. We assume that $ f $ admits a section $ \sigma $. Let $ S $ be the subset of $ C $ consisting of the points $ s $ such that the fiber $ X_s $ is singular. Let $ m_s $ be the number of irreducible components of $ X_s $. Let $ U = C - S $ and $ X^\circ = f^{-1}(U) $. Let $ g $ denote the genus of a nonsingular fiber. 

We analyze the Hodge structure on $ H^2(X, \QQ) $:

In \cite{zucker}, Zucker proves that the Leray spectral sequence for $ f $ degenerates on page 2 \cite[Cor. 15.15]{zucker} and respects the Hodge structures. Let the Leray filtration on $ H^2(X, \QQ) $ be
\[ H^2(X, \QQ) = L^0 \supseteq L^1 \supseteq L^2 \supseteq L^3 = 0. \]
Then we have:
\begin{enumerate}
\item $ L^0/L^1 \cong H^0(C, R^2f_*\QQ) $. 
\item $ L^1/L^2 \cong H^1(C, R^1f_*\QQ) $. This is the ``mysterious'' piece. 
\item $ L^2/L^3 \cong H^2(C, f_*\QQ) = H^2(C, \QQ) \cong \QQ $, since the fibers are connected. 
\end{enumerate} 
Each of these pieces are Hodge structures of weight $ 2 $. 

\begin{lemma}
  \[ R^2f_*\QQ = \QQ \bigoplus M, \] where $ M $ is a skyscraper sheaf with stalks $ M_s = \QQ^{m_s-1} $ for $ s \in S $. 
\end{lemma}
\begin{proof}
  There is a map $ \eta_2 : R^2f_*\QQ \to j_*j^{-1}R^2f_*\QQ $ - the unit of adjunction. Clearly, the two sheaves coincide on $ U $ under $ \eta_2 $. For $ s \in S $, we have $ (R^2f_*\QQ)_s = H^2(X_s, \QQ) $ and $ (j_*j^{-1}R^2f_*\QQ)_s $ consists of the local invariant cycles around $ s $. By the local invariant cycle theorem, the map $ (R^2f_*\QQ)_s \to (j_*j^{-1}R^2f_*\QQ)_s $ is surjective. Next, $ j_*j^{-1}R^2f_*\QQ = j_*\QQ $, since the stalks are all isomorphic to $ \QQ $, and the nonsingular fibers are all oriented compact Riemann surfaces. It follows that the kernel of $ \eta_2 $ is \[ M = \bigoplus_{s \in S} \QQ^{m_s-1}. \]
\end{proof}

\begin{proposition}
  $ L^0/L^1 $ is spanned by the cohomology classes of the divisors. 
\end{proposition}
\begin{proof}
  $ L^0/L^1 = H^0(C, R^2f_*\QQ) = H^0(C, \QQ) \bigoplus H^0(C, M) \cong \QQ \oplus \bigoplus_{s \in S} \QQ^{m_s-1}$. The classes in $ H^0(C, \QQ) \cong \QQ $ are spanned by the image of the section $ \sigma $, and the classes in $ H^0(C, M) = \bigoplus_{s \in S} \QQ^{m_s-1} $ are spanned by the irreducible components of $ X_s $ (except the multiples of $ X_s $ itself). 
\end{proof}
Also, 
\begin{proposition}
  $ L^2/L^3 $ is spanned by the cohomology class of a fiber. 
\end{proposition}
\begin{proof}
  The inclusion $ f^*H^2(C, \QQ) \subset H^2(X, \QQ) $ sends the cohomology class of a point on $ C $ to the cohomology class of a fiber. 
\end{proof}

Clearly, all the classes in $ L^0/L^1 $ and $ L^2/L^3 $ are of type $ (1,1) $ with respect to the Hodge structure on $ H^2(X, \QQ) $. Next, we deal with the mysterious piece $ H^1(C, R^1f_*\QQ) $. 

\begin{lemma}
  $ R^1f_*\QQ = j_*j^{-1}R^1f_*\QQ $. 
\end{lemma}

\begin{proof}
  The unit of adjunction $ \eta_1: R^1f_*\QQ \to j_*j^{-1}R^1f_*\QQ $ is surjective because of the local invariant cycle theorem, exactly as above. It is also injective \cite[Cor. 13.4.3]{arapura}. 
\end{proof}

We call $ f : X \to C $ {\em extremal} if $ H^1(C, j_*j^{-1}R^1f_* \CC)^{1,1} $ vanishes.

\begin{proposition}
  If $ f : X \to C $ is extremal, then $ X $ is Picard maximal. 
\end{proposition}

\begin{proof}
  By propositions 2 and 3, if $ H^1(C, j_*j^{-1}R^1f_* \CC)^{1,1} $ is spanned by divisors, then $ X $ is Picard maximal. In particular, if $ H^1(C, j_*j^{-1}R^1f_* \CC)^{1,1} = 0 $, then $ X $ is Picard maximal. 
\end{proof}

Here we have given a Hodge structure to $ H^1(C, j_*j^{-1}R^1f_*\QQ) $, via the Leray filtration on $ H^2(X, \QQ) $, in an extrinsic way. We note here that Zucker constructs a Hodge structure of weight $ 2 $ on $ H^1(C, j_*j^{-1}R^1f_*\QQ) $ in an intrinsic manner, and that the two Hodge structures coincide. We recall some of the results of Deligne regarding variations of Hodge structure, see \cite[Section 2]{zucker}: 

Let $ V $ be a variation of Hodge structure of weight $ m $ on a smooth projective variety $ S $. Let
\[ \OOO_{S}(V) = \sF^{0} \supset \cdots \sF^{m} \supset \sF^{m+1} = 0 \]
denote the Hodge subbundles of $ \OOO_{S}(V) $. The de Rham resolution $ V \to \Omega_{S}^{\bullet}(V) $ admits a filtration given by
\[ F^{p}\Omega_{S}^{\bullet}(V) = \sF^{p} \to \sF^{p-1} \otimes \Omega_{s}^{1} \to \sF^{p-2} \otimes \Omega_{s}^{2} \to \cdots \]
\begin{theorem}[Deligne]
  $ H^{i}(S, V) $, regarded as $ \HHH^{i}(S, \Omega_{S}^{\bullet}(V)) $, is a Hodge structure of weight $ i+m $, with the Hodge filtration given by:
  \[ F^{p}H^{i}(S, V) := \HHH^{i}(F^{p}\Omega_{S}^{\bullet}(V)), \]
  and
  \[ \Gr_{F}^{p}H^{i}(S,V) = \HHH^{i}(\Gr_{F}^{p}\Omega_{S}^{\bullet}(V)). \]
\end{theorem}

\section{Elliptic curves}
We recall the construction of the universal family of elliptic curves over the upper half plane to fix the notation. Let $ \HH $ denote the upper half plane, and let $ \tau \in \HH $. Let $ \Gamma_{\tau} $ denote the lattice $ \ZZ + \ZZ\cdot\tau \subset \CC $. Let $ \EE_{\tau} $ denote the elliptic curve $ \CC/\Gamma_{\tau} $. Let $ h : \EE \to \HH $ be the family of elliptic curves over the upper half plane, consisting of all the $ \EE_{\tau} $. Note that $ h : \EE \to \HH $ is the \emph{universal} family of the following data: elliptic curves $ E $ together with a symplectic basis for $ H_{1}(E,\ZZ) $. We record some observations for later use: 

Consider the variation of Hodge structure of weight $ 1 $ on $ \HH $ given by:
\[ W_{\QQ} := R^{1}h_{*}\QQ \]
Let $ \sF^{\bullet} $ denote the Hodge filtration on $ \WW := W_{\QQ} \otimes_{\QQ} \OOO_{\HH} $. Then $ W_{\CC} $ is resolved by the complex
\[\Omega_{\HH}^{\bullet}(\WW) := \WW \to \WW \otimes \Omega^{1}_{\HH} \] 
The Hodge filtration induces a filtration on this complex, as follows: 
\[F^{0}\Omega^{\bullet}(\WW) := \WW \to \WW \otimes \Omega^{1}_{\HH} \]
\[F^{1}\Omega^{\bullet}(\WW) := \sF^{1} \to \WW \otimes \Omega^{1}_{\HH} \]
\[F^{2}\Omega^{\bullet}(\WW) := 0 \to \sF^{1} \otimes \Omega^{1}_{\HH} \]

We observe the following for later use: 
\begin{proposition}\label{elliptic-extremality}
  The complex $ \Gr_{\sF}^{1}\Omega_{\HH}^{\bullet}(\WW) $ is quasi-isomorphic to zero. 
\end{proposition}
\begin{proof}
  The complex in question is $ \Gr_{\sF}^{1} \to \Gr_{\sF}^{0} \otimes \Omega_{\HH}^{1} $. Both pieces are rank one vector bundles and the differential is an isomorphism: this follows from the nondegeneracy of the Kodaira-Spencer class for the family $ h : \EE \to \HH $. 
\end{proof}

\section{Quaternionic examples}
We refer to Rotger's thesis \cite[Chapter 4]{rotger} for details about quaternions and abelian varieties with quaternionic multiplication. Let $ B $ be an indefinite quaternion division algebra over $ \QQ $. In concrete terms, $ B $ is a division algebra over $ \QQ $, which is spanned by $ 1,i,j,ij $; subject to the relations $ i^{2} = a, j^{2} = b $ and $ ij = -ji $ for some $ a,b \in \QQ^{\times} $ and such that at least one of $ a $ and $ b $ is positive. Fix an isomorphism
\[ \eta : B \otimes \RR \cong M_{2}(\RR). \]
The set of primes $ p $ such that $ B \otimes_{\QQ}\QQ_{p} \not\cong M_{2}(\QQ_{p}) $ is finite and has even cardinality. The product of all these primes is called the \emph{discriminant} of $ B $. Let $ D $ denote the discriminant of $ B $. Let $ \OO $ be a maximal order in $ B $, that is, $ \OO $ is a maximal element among rank $ 4 $ subrings $ A $ of $ B $ such that $ A \otimes_{\ZZ} \QQ = B $. Let $ \mu \in \OO $ be a pure quaternion such that $ \mu^{2} = -D $. Such a $ \mu $ actually exists, see \cite[Section 4.6.1]{rotger}. Let $ \OO^{1} $ denote the group of elements of $ \OO $ with reduced norm $ 1 $. Since $ B $ is a division algebra, the quotient $ \HH/\eta(\OO^{1}) $ is compact, see \cite[Chapter 9, Section 2]{shimura}.

Fix $ B, \OO $ and $ \mu $ as above. Consider the skew symmetric form
\[ E:\OO \otimes \OO \to \ZZ \]
given by
\[ E(\alpha,\beta) := \frac{-1}{D}\tr(\mu\alpha\bar\beta) = \frac{1}{D}\tr(\mu\bar\alpha\beta) \]
If $ u \in \OO^{1} $, then $ E(u\alpha, u\beta) = E(\alpha,\beta) $ for all $ \alpha, \beta \in \OO $. Therefore $ \OO^{1} < \Sp(E) $, where $ \Sp(E) $ denotes the symplectic group of the skew symmetric form $ E $. Note that after changing coefficients to $ \RR $, the inclusion $ \OO^{1} < \Sp(E) $ extends to

\[\begin{CD}
  \OO^{1} @>>> \Sp(E)\\
  @V{\eta}VV @VVV\\
  \Sp_{2}(\RR) @>>> \Sp_{4}(\RR)
\end{CD}\]

The image of $ \OO^{1} $ in $ \Sp_{4}(\RR) $ lands inside $ \Sp_{4}(\ZZ) $. Let $ G_{n} $ denote the kernel of the map $ \OO^{1} \to \Sp_{4}(\ZZ) \to \Sp_{4}(\ZZ/n\ZZ) $. Note that $ G_{n} $ is a finite index subgroup of $ \OO^{1} $. 

Now we construct a family of principally polarized abelian surfaces over $ \HH $ with quaternionic multiplication by $ \OO $. Let $ \Lambda_{\tau} $ be the lattice in $ \CC^{2} $ given by
\[ \Lambda_{\tau} := \eta(\OO)\cdot \left( \begin{array}{c} \tau \\ 1\\ \end{array} \right), \]
where $ \tau \in \HH $. Let $ \AAA_{\tau} $ denote the complex torus $ \CC^{2}/\Lambda_{\tau} $. Let $ E_{\tau} $ denote the skew symmetric form on $ \Lambda_{\tau} $ induced by $ E $ and also its $ \RR $-linear extension to $ \CC^{2} $. The skew symmetric form $ E_{\tau} $ determines a principal polarization on $ \AAA_{\tau} $. Let $ \AAA \to \HH $ denote the family over the upper half plane formed by all the $ \AAA_{\tau} $ for $ \tau \in \HH $. This is the universal family of principally polarized abelian surfaces $ A $ with quaternionic multiplication by $ \OO $, together with a symplectic basis for $ H_{1}(A,\ZZ) $. Let $ \XX $ denote the universal theta divisor on $ \AAA $, defined as the vanishing locus of the Riemann theta function \cite[Section 1]{grushevsky-hulek}.

\begin{proposition}
  $ \XX_{\tau} $ is either a smooth genus two curve or the union of two smooth isogenous genus one curves meeting transversely at a point. 
\end{proposition}

\begin{proof}
$ \AAA_{\tau} $ is either a simple abelian surface or isogenous to a product $ E_{1} \times E_{2} $ of two elliptic curves. In the former case, $ \XX_{\tau} $ must be a nonsingular genus $ 2 $ curve. In the latter case, $ \AAA_{\tau} \sim E_{1} \times E_{2} $, the elliptic curves $ E_{1} $ and $ E_{2} $ must be isogenous, since otherwise the endomorphism algebras will satisfy:
\[ B \subseteq \End_{\QQ}(\AAA_{\tau}) \cong \End_{\QQ}(E_{1}) \oplus \End_{\QQ}(E_{2}). \]
But this cannot happen, as $ B $ is noncommutative and $ \End_{\QQ}(E_{1}) \oplus \End_{\QQ}(E_{2}) $ is commutative. Also note that if $ \AAA_{\tau} \sim E^{2} $ for an elliptic curve $ E $, the endomorphism algebras must satisfy:
\[ B \subseteq \End_{\QQ}(\AAA_{\tau}) \cong M_{2}(\End_{\QQ}(E)) \]
and therefore $ B $ must contain the imaginary quadratic field $ \End_{\QQ}(E) $ as a $ \QQ $-subalgebra. Now the assertion follows from \cite[Theorem 3.8]{grushevsky-hulek}. (It isn't yet clear to me whether the theta divisor in this case really a union of two elliptic curves or a smooth genus $ 2 $ curve, see the comment after \cite[Corollary 10.6.3]{birkenhake-lange}.)
\end{proof}

Let $ g_{n} : A_{n} \to C_{n} $ denote the quotient of $ g : \AAA \to \HH $ by the group $ G_{8n} $ and let $ f_{n} : X_{n} \to C_{n} $ be the quotient of $ \XX \to \HH $ by $ G_{8n} $ for $ n \geq 1 $. This is well defined, see \cite[Section 1]{grushevsky-hulek}. Also, it admits a (multi-)section given by tracing a Weierstrass point of a general $ \XX_{\tau} $. 

\begin{proposition}
  $ R^{1}g_{n*}\QQ = j_{*}j^{-1}R^{1}f_{n*}\QQ $. 
\end{proposition}
\begin{proof}
  This follows from the fact that $ \AAA_{\tau} $ is the Albanese of $ \XX_{\tau} $. 
\end{proof}

\begin{theorem}
  $ f_{n} : X_{n} \to C_{n} $ is extremal, and hence $ X_{n} $ is Picard maximal. 
\end{theorem}
\begin{proof}
  Because of the previous proposition, it suffices to show that
  \[ H^{1}(C_{n}, R^{1}g_{n*}\QQ)^{1,1} = 0. \]
  By Theorem 6, this group is the hypercohomology of the complex $ Gr^{1}\Omega_{C_{n}}^{\bullet}(R^{1}g_{n*}\QQ) $. We will show that the complex $ Gr^{1}\Omega_{C_{n}}^{\bullet}(R^{1}g_{n*}\QQ) $ is quasi-isomorphic to zero on $ C_{n} $. Since this statement is local analytic on the base, it suffices to prove that
  \[ Gr^{1}\Omega_{\HH}^{\bullet}(R^{1}g_{*}\QQ) \]
  is quasi-isomorphic to zero on $ \HH $. Let $ V_{\QQ} $ denote the variation of Hodge structure $ R^{1}g_{*}\QQ $ on $ \HH $. We claim that
  \[ V_{\RR} = W_{\RR}^{\oplus 2} \]
  as variations of real Hodge structures. To show this, we calculate the real Hodge structures $ H^{1}(\AAA_{\tau}, \RR) $ and $ H^{1}(\EE_{\tau}, \RR)^{\oplus 2} $ and show that they are equal:

  For $ \EE_{\tau} $:
  \[ H^{1}(\EE_{\tau}, \CC)^{\oplus 2} = \Hom_{\ZZ}(\Gamma_{\tau} , \CC)^{\oplus 2} = \Hom_{\RR}(\Gamma_{\tau}^{\oplus 2} \otimes_{\ZZ} \RR , \CC), \]
  and consequently the Hodge decomposition is:
  \[ \Hom_{\RR}(\Gamma_{\tau}^{\oplus 2} \otimes_{\ZZ} \RR , \CC) = \Hom_{\CC}(\Gamma_{\tau}^{\oplus 2} \otimes_{\ZZ} \RR , \CC) \oplus \Hom_{\bar\CC}(\Gamma_{\tau}^{\oplus 2} \otimes_{\ZZ} \RR , \CC), \]
  where on the right side, the homomorphisms are with respect to the complex structure coming from $ \Gamma_{\tau} \subset \CC $. 
  Similarly for $ \AAA_{\tau} $, 
  \[ H^{1}(\AAA_{\tau}, \CC) = \Hom_{\ZZ}(\Lambda_{\tau} , \CC) = \Hom_{\RR}(\Lambda_{\tau} \otimes_{\ZZ} \RR , \CC), \]
  and the Hodge decomposition:
  \[ \Hom_{\RR}(\Lambda_{\tau} \otimes_{\ZZ} \RR , \CC) = \Hom_{\CC}(\Lambda_{\tau} \otimes_{\ZZ} \RR , \CC) \oplus \Hom_{\bar\CC}(\Lambda_{\tau} \otimes_{\ZZ} \RR , \CC), \]
  where on the right side, the homomorphisms are with respect to the complex structure coming from $ \Lambda_{\tau} \subset \CC^{2} $. 
  Now observe that $ \Gamma_{\tau}^{\oplus 2} \otimes_{\ZZ} \RR = M_{2}(\RR) $ and $ \Lambda_{\tau} \otimes_{\ZZ} \RR = M_{2}(\RR) $ with respect to the canonical symplectic bases of $ \Gamma_{\tau} $ and $ \Lambda_{\tau} $. Therefore the two Hodge structures are equal under the implicit canonical identifications. Thus the result follows, by proposition \ref{elliptic-extremality}. 
\end{proof}

\begin{proposition}\label{generaltype}
  $ X_{n} $ is a surface of general type and $ p_{g}(X_{n}) > 0 $ for all large enough $ n $. 
\end{proposition}

\begin{proof}
  For large enough $ n $, the curve $ C_{n} $ is smooth and has genus at least two. The bundle $ \omega_{X_{n}/C_{n}} $ admits global sections given by Siegel modular theta functions, so $ p_{g}(X_{n}) > 0 $, see \cite[p.411]{igusa}. (Note that our $ G_{8n} $ is contained in Igusa's $ \Gamma(4,8) $). Iitaka's conjecture $ C_{2,1} $ regarding the subadditivity of Kodaira dimension implies that $ X_{n} $ is a surface of general type, see \cite[p.197]{viehweg}. 
\end{proof}

\subsection*{Note}
The construction in \cite{arapura-solapurkar} may be regarded as a degenerate form of the above construction, in case of the quaternion algebra $ M_{2}(\QQ) $. There we had to go through some adjustments to $ \EE \times_{\HH} \EE $ before passing to the fiberwise theta divisor. 



\end{document}